\documentclass[12pt]{article}
\usepackage{amssymb,amsfonts,amsmath,amsthm,latexsym, epsfig,mathrsfs,colordvi}
\usepackage{graphicx}
\newtheorem{thm}{Theorem}
\newtheorem{lemma} [thm]{Lemma}
\newtheorem{cor}[thm]{Corollary}
\newtheorem{prop}[thm]{Proposition}

\theoremstyle{definition}

\newtheorem{example}[thm]{Example}

\newcommand{\beas}{\begin{eqnarray*}}
\newcommand{\eeas}{\end{eqnarray*}}
\newcommand{\bea}{\begin{eqnarray}}
\newcommand{\eea}{\end{eqnarray}}
\newcommand{\be}{\begin{enumerate}}
\newcommand{\ee}{\end{enumerate}}
\newcommand{\beq}{\begin{equation}}
\newcommand{\eeq}{\end{equation}}
\newcommand{\sn}{\mathfrak{S}_n}
\newcommand{\fs}{\mathfrak{S}}
\newcommand{\cc}{\mathbb{C}}
\newcommand{\zz}{\mathbb{Z}}

\newcommand{\fra}{\cc\{\{x\}\}}
\newcommand{\afra}{\cc_{\mathrm{alg}}\{\{x\}\}}
\newcommand{\ec}{\langle 12\cdots n\rangle}
\newcommand{\bm}[1]{{\mbox{\boldmath $#1$}}}

\makeatletter \@addtoreset{equation}{section}
\@addtoreset{theo}{}\makeatother

\begin{document}

\title{An Equivalence Relation on the Symmetric Group and
  Multiplicity-free Flag $h$-Vectors}

\author{Richard P.
  Stanley\footnote{ Department of Mathematics, Massachusetts Institute
    of Mathematics, Cambridge, MA 02139, USA. Email:
    rstan@math.mit.edu. The
    author's contribution is based upon work supported by the National
    Science Foundation under grants nos.\ DMS-0604423 and DMS-1068625.}}

\date{August 15, 2012}

\maketitle



\begin{abstract}
We consider the equivalence relation $\sim$ on the symmetric group 
$\sn$ generated by the interchange of two adjacent elements $a_i$ and
$a_{i+1}$ of $w=a_1 \cdots a_n\in\sn$ such that $|a_i-a_{i+1}|=1$. We
count the number of equivalence classes and the sizes of the
equivalence classes. The results are generalized to permutations of
multisets. In the
original problem, the equivalence class containing the identity
permutation is the set of linear extensions of a certain
poset. Further investigation yields a characterization of all finite
graded posets whose flag $h$-vector takes on only the values $0,\pm
1$. 
\end{abstract}



\section{Introduction.} \label{sec1} 
Let $k\geq 2$. Define two permutations $u$ and $v$ (regarded as words
$a_1 a_2\cdots a_n$) in the symmetric group $\sn$ to be
\emph{equivalent} if $v$ can be obtained from $u$ by a sequence of
interchanges of adjacent terms that differ by at least $j$. It is a
nice exercise to show that the number $f_j(n)$ of equivalence classes
of this relation (an obvious equivalence relation) is given by
  $$ f_j(n) = \left\{ \begin{array}{rl} n!, & n\leq j\\[.5em]
                j!\cdot j^{n-j}, & n>j. \end{array} \right. $$
Namely, one can show that every equivalence class contains a unique
permutation $w=b_1b_2\cdots b_n$ for which we never have $b_i \geq
b_{i+1}+j$. To count these permutations $w$ for $n>j$, we first have
$j!$ ways of ordering $1,2,\dots,j$ within $w$. Then insert $j+1$ in
$j$ ways, i.e., at the end or preceding any $i \neq 1$. Next insert
$j+2$ in $j$ ways, etc. The case $j=3$ of this argument appears in
\cite[A025192]{oeis} and is attributed to Joel Lewis, November 14,
2006. Some equivalence relations on $\sn$ of a similar nature are
pursued by Linton et al.\ \cite{linton}.

The above result suggests looking at some similar equivalence
relations on $\sn$. The one we will consider here is the following:
define $u$ and $v$ to be \emph{equivalent}, denoted $u\sim v$, if $v$
can be obtained from $u$ by interchanging adjacent terms that differ
by exactly one. For instance, when $n=3$ we have the two equivalence
classes $\{123, 213, 312\}$ and $\{321,231, 312\}$. 
We will determine
the number of classes, the number of one-element classes, and the
sizes of the equivalence classes (always a product of Fibonacci
numbers). It turns out that the class containing the identity
permutation $12\cdots n$ may be regarded as the set of linear
extensions of a certain $n$-element poset $P_n$. Moreover, $P_n$ has
the most number of linear extensions of any $n$-element poset on the
vertex set $[n]=\{1,2,\dots,n\}$ such that $i<j$ in $P$ implies $i<j$
in $\zz$, and such that all linear extensions of $P$
(regarded as permutations of $[n]$) have a different descent set. This
result leads to the complete classification and enumeration of finite
graded posets of rank~$n$ whose flag $h$-vector is
``multiplicity-free,'' i.e., assumes only the values 0 and $\pm 1$.

\section{The number of equivalence classes.} 
To obtain the number of equivalence classes, we first define a
canonical element in each class. We then count these canonical
elements by the Principle of Inclusion-Exclusion. We call a
permutation $w=a_1a_2\cdots a_n\in\sn$ \emph{salient} if we never have
$a_i=a_{i+1}+1$ ($1\leq i\leq n-1$) or $a_i=a_{i+1}+2= a_{i+2}+1$
($1\leq i\leq n-2$). For instance, there are eight salient
permutations in $\fs_4$: 1234, 1342, 2314, 2341, 2413, 3142, 3412,
4123. 

\begin{lemma} 
\label{lemma1}
Every equivalence class with respect to the equivalence relation
$\sim$ contains exactly one salient permutation. 
\end{lemma}

\begin{proof}
Let $E$ be an equivalence class, and let $w=a_1a_2\cdots a_n$ be the
lexicographically least element of $E$. Then we cannot have
$a_i=a_{i+1}+1$ for some $i$; otherwise we could interchange $a_i$ and
$a_{i+1}$ to obtain a lexicographically smaller permutation in
  $E$. Similarly if $a_i=a_{i+1}+2=a_{i+2}+1$ then we can replace $a_i
  a_{i+1} a_{i+2}$ with $a_{i+2}a_ia_{i+1}$. Hence $w$ is salient.

It remains to show that a class $E$ cannot contain a salient
permutation $w=b_1b_2\cdots b_n$ that is not the lexicographically least
element $v=a_1a_2\cdots a_n$ of $E$. Let $i$ be the least index for
which $a_i\neq b_i$. Since $v\sim w$, there must be some $b_j$
satisfying $j>i$ and $b_j<b_i$ that is interchanged with
$b_i$ in the transformation of $v$ to $w$ by adjacent transpositions
of consecutive integers. Hence $b_j=b_i+1$. If $j=i+1$ then $v$ is not
salient. If $j=i+2$ then we must have $b_{i+1}=b_i+2$ in order to move
$b_{j+2}$ past $b_{j+1}$, so again $v$ is not salient. If $j>i+2$
then some element $b_k$ between $b_i$ and $b_j$ in $v$ satisfies
$|b_j-b_k|>1$, so we cannot move $b_j$ past $b_k$ unless we first
interchange $b_i$ and $b_k$ (which must therefore equal $b_i+1$). But
then after $b_i$ and $b_j$ are interchanged, we cannot move $b_k$ back
to the right of $b_i$. Hence $v$ and $w$ cannot be equivalent, a
contradiction completing the proof.
\end{proof}

\begin{thm} \label{thm:fn}
Let $f(n)$ be the number of equivalence classes of the relation $\sim$
on $\sn$, with $f(0)=1$.  Then
  \beq f(n) = \sum_{j=0}^{\lfloor
    n/2\rfloor}(-1)^j(n-j)!\binom{n-j}{j}. \label{eq:fn} \eeq
Equivalently,
   $$ \sum_{n\geq 0}f(n)x^n = \sum_{m\geq 0}m!(x(1-x))^m. $$
\end{thm}

\begin{proof} 
By Lemma~\ref{lemma1}, we need to count the number of salient
permutations $w\in\sn$. The proof is by an inclusion-exclusion
argument. Let $A_i$, $1\leq i\leq n-1$,  be the set of permutations
$v\in\sn$ that contain the factor (i.e., consecutive terms) $i+1,
i$. Let $B_i$, $1\leq i\leq n-2$, be the set of $v\in\sn$ that contain
the factor $i+2,i,i+1$. Let $C_1,\dots,C_{2n-3}$ be some indexing of
the $A_i$'s and $B_i$'s. By the Principle of Inclusion-Exclusion, we
have
  \beq f(n)=\sum_{S\subseteq [2n-3]}(-1)^{\#S}\#\bigcap_{i\in S}C_i, 
   \label{eq:incex} \eeq
where the empty intersection of the $C_i$'s is $\sn$. A little thought
shows that any intersection of the $C_i$'s consists of permutations
that contain some set of nonoverlapping factors $j,j-1,\dots,i+1,i$ and
$j,j-1,\dots,i+3,i+2,i,i+1$. Now permutations containing the factor
$j,j-1,\dots,i+1,i$ are those in $A_{j-1}\cap A_{j-2}\cap\cdots\cap
A_i$ (an intersection of $j-i$ sets), while permutations containing
$j,j-1,\dots,i+3,i+2,i,i+1$ are those in $A_{j-1}\cap A_{j-2}\cap
\cdots\cap A_{i+2}\cap B_i$ (an intersection of $j-i-1$ sets). Since
$(-1)^{j-i} +(-1)^{j-i-1}=0$, it follows that all terms on the
right-hand side of equation~\eqref{eq:incex} involving such
intersections will cancel out. The only
surviving terms will be the intersections $A_{i_1}\cap\cdots\cap
A_{i_j}$ where the numbers $i_1,i_1+1,i_2,i_2+1,\dots, i_j,i_j+1$
are all \emph{distinct}. The number of ways to choose such terms for a
given $j$ is the number of sequences of $j$ 2's and $n-2j$ 1's, i.e.,
$\binom{n-j}{2j}$. The number of permutations of the $j$ factors
$i_r,i_{r+1}$, $1\leq r\leq j$, and the remaining $n-2j$ elements of
$[n]$ is $(n-j)!$. Hence equation~\eqref{eq:incex} reduces to
equation~\eqref{eq:fn}, completing the proof.
\end{proof}

\textsc{Note.} There is an alternative proof based on the
Cartier-Foata theory of partially commutative monoids \cite{c-f}.
Let $M$ be the monoid with generators $g_1,\dots,g_n$ subject only to
relations of the form $g_ig_j=g_jg_i$ for certain $i$ and $j$. Let
$x_1,\dots,x_n$ be \emph{commuting} variables. If $w=g_{i_1}\cdots
g_{i_m}\in M$, then set $x^w = x_{i_1}\cdots x_{i_m}$. Define
  $$ F_M(x) = \sum_{w\in M} x^w. $$
Then a fundamental result (equivalent to
\cite[Thm.~2.4]{c-f}) of the theory asserts that 
  \beq F_M(x) = \frac{1}{\sum_S (-1)^{\#S}\prod_{g_i\in S}
    x_i}, \label{eq:cf} \eeq
where $S$ ranges over all subsets of $\{g_1,\dots,g_n\}$ (including
the empty set) whose elements pairwise commute. Consider now the case
where the relations are given by $g_i g_{i+1}=g_{i+1}g_i$, $1\leq
i\leq n-1$. Thus $f(n)$ is the coefficient of $x_1 x_2\cdots x_n$ in 
$F(x)$. Writing $[x^\alpha]G(x)$ for the coefficient of $x^\alpha =
x_1^{\alpha_1}\cdots x_n^{\alpha_n}$ in the power series $G(x)$, it
follows from equation~\eqref{eq:cf} that
  \beas f(n) & = & [x_1\cdots x_n]\frac{1}{1-\sum_{i=1}^n x_i
        +\sum_{i=1}^{n-1} x_ix_{i+1}}\\ & = & 
     [x_1\cdots x_n]\sum_{j\geq 0} \left( \sum_{i=1}^n x_i
        -\sum_{i=1}^{n-1} x_ix_{i+1}\right)^j. \eeas
A straightforward argument shows that 
   $$ [x_1\cdots x_n] \left( \sum_{i=1}^n x_i
        -\sum_{i=1}^{n-1} x_ix_{i+1}\right)^{n-j} =
     (-1)^j(n-j)!\binom{n-j}{j}, $$
and the proof follows. 

\textsc{Note.} The numbers $f(n)$ for $n\geq 0$ begin 1, 1,
1, 2, 8, 42, 258, 1824, 14664, $\dots$. This sequence appears in
\cite[A013999]{oeis} but without a combinatorial interpretation before
the present paper.

Various generalizations of Theorem~\ref{thm:fn} suggest
themselves. Here we will say a few words about the situation where
$\sn$ is replaced by all permutations of a \emph{multiset} on the set
$[n]$ with the same definition of equivalence as
before. For instance, for the multiset $M=\{1^2,2,3^2\}$ (short for
$\{1,1,2,3,3\}$), there are six equivalence classes, each with five
elements, obtained by fixing a word in $1,1,3,3$ and inserting 2 in
five different ways. Suppose that the multiset is given by
$M=\{1^{r_1},\dots, n^{r_n}\}$. According to equation~\eqref{eq:cf}, the 
number $f_M$ of equivalence classes of permutations of $M$ is the
coefficient of $x_1^{r_1}\cdots x_n^{r_n}$ in the generating function
  $$ F_n(x) = \frac{1}{1-\sum_{i=1}^n x_i+\sum_{i=1}^{n-1}
    x_ix_{i+1}}. $$  
For the case $n=4$ (and hence $n\leq 4$) we can give an explicit
formula for the coefficients of $F_n(x)$, or in fact (as suggested by
I. Gessel) for $F_n(x)^t$ where $t$ is an indeterminate. We use the
falling factorial notation $(y)_r =y(y-1)\cdots (y-r+1)$.

\begin{thm}\label{thm:k4}
We have
 \beq F_4(x)^t = \sum_{h,i,j,k\geq 0} \frac{(t+h+j-1)_j (t+h+k-1)_h
   (t+i+k-1)_{i+k}}{h!\,i!\,j!\,k!}x_1^hx_2^ix_3^kx_4^k. 
   \label{eq:f4xt} \eeq
In particular,
 \beq F_4(x) = \sum_{h,i,j,k\geq 0}\binom{h+j}{j}\binom{h+k}{k}
   \binom{i+k}{i} x_1^h  x_2^i x_3^j x_4^k. \label{eq:f4} \eeq
\end{thm}

\begin{proof}
Since the coefficient of $x_1^h  x_2^i x_3^j x_4^k$ in $F_4(x)^t$ is a
polynomial in $t$, it suffices to assume that $t$ is a nonnegative
integer. The result can then be proved straightforwardly by induction
on $t$. Namely, the case $t=0$ is trivial. Assume for $t-1$ and let
$G(x)$ be the right-hand side of equation~\eqref{eq:f4xt}. Check that
  $$ (1-x_1-x_2-x_3-x_4+x_1x_2+x_2x_3+x_3x_4)G(x) = F_4(x)^{t-1}, $$
and verify suitable initial conditions.

Ira Gessel points out (private communication) that we can prove the
theorem without guessing the answer in advance by writing
  $$ \frac{1}{1-x_1-x_2-x_3-x_4+x_1x_2+x_2x_3+x_3x_4} = $$
  $$  \frac{1}{(1-x_2)(1-x_3)\left( \displaystyle 1-\frac{x_1}{1-x_3}
    \right)\left(1-\displaystyle \frac{x_4}{(1-x_2)\left(
    1-\frac{x_1}{1-x_3}\right)} \right)}, $$
and then expanding one variable at a time in the order $x_4, x_1, x_2,
x_3$. 
\end{proof}

For multisets supported on sets with more than four elements there are
no longer simple explicit formulas for the number of equivalence
classes. However, we can still say something about the multisets
$\{1^k, 2^k, \dots, n^k\}$ for $k$ fixed or $n$ fixed. The simplest
situation is when $n$ is fixed.

\begin{prop} \label{prop:kn}
Let $g(n,k)$ be the number of equivalence classes of permutations of
the multiset $\{1^k,\dots,n^k\}$. For fixed $n$, $g(n,k)$ is a
$P$-recursive function of $k$, i.e., for some integer $d\geq 1$ and
polynomials $P_0(k),\dots,P_d(k)$ (depending on $n$), we have
  $$ P_0(k)g(n,k+d)+P_1(k)g(n,k+d-1)+\cdots+P_d(k)g(n,k)=0 $$
for all $k\geq 0$. 
\end{prop}

\begin{proof}
The proof is an immediate consequence of equation~\eqref{eq:cf}, the
result of Lipshitz \cite{lip} that the diagonal of a rational
function (or even a $D$-finite function) is $D$-finite, and the
elementary result 
\cite[Thm.~1.5]{rs:df}\cite[Prop.~6.4.3]{ec2} that the coefficients of
$D$-finite series are $P$-recursive. 
\end{proof}

To deal with permutations of the multiset $\{1^k,\dots,n^k\}$ when $k$
is fixed, let $\fra$ denote the field of fractional Laurent series
$f(x)$ over $\cc$ with finitely many terms having a negative exponent,
i.e., for some $j_0\in\zz$ and some $N\geq 1$ we have $f(x) =
\sum_{j\geq j_0} a_jx^{j/N}$, $a_j\in\cc$. A series $y\in \fra$ is
\emph{algebraic} if it satisfies a nontrivial polynomial equation
whose coefficients are polynomials in $x$. Any such polynomial
equation of degree $n$ has $n$ zeros (including multiplicity)
belonging to the field $\fra$. In fact, this field is algebraically
closed (Puiseux' theorem). For further information, see for instance
\cite[{\S}6.1]{ec2}. Write $\afra$ for the field of algebraic
fractional (Laurent) series.  The next result is a direct
generalization of Theorem~\ref{thm:fn}. The main point is that the
series $z_i(x)$ and $y_j(x)$ are \emph{algebraic}.

\begin{thm} \label{thm:umbral}
Let $k$ be fixed. Then there exist finitely many
algebraic fractional series 
$y_1,\dots,y_q$, $z_1,\dots, z_q \in \afra$  
and polynomials $P_1,\dots,P_q\in\afra[m]$ (i.e.,
polynomials in $m$ whose coefficients lie in $\afra$) such that 
  $$ \sum_{n\geq 0}g(n,k)x^n =\sum_{j=1}^q z_j(x)\sum_{m\geq 0}
       m!P_j(m)y_j(x)^m. $$  
\end{thm}

\begin{proof}
Our proof will involve ``umbral'' methods. By the work of Rota et al.\
\cite{roman}\cite{rota}, this means that we will be dealing with
polynomials in $t$ (whose coefficients will be fractional Laurent
series in $x$) and will apply a linear functional $\varphi\colon
\cc[t]\{\{x\}\} \to \fra$. For our situation $\varphi$ is defined by
$\varphi(t^m)=m!$.

\textsc{Note.} The ring $\cc[t]\{\{x\}\}$ consists of all series of
the form $\sum_{j\geq j_0}a_j(t)x^{j/N}$ for some $j_0\in\zz$ and
$N\geq 1$, where $a_j(t)\in\cc[t]$. When we replace $t^m$ with $m!$,
the coefficient of each $x^{j/N}$ is a well-defined complex
number. The function $\varphi$ is not merely linear; it commutes with 
\emph{infinite} linear combinations of the form  $\sum_{j\geq
  j_0}a_j(t)x^{j/N}$. In other
words, $\varphi$ is \emph{continuous} is the standard topology on
$\cc[t]\{\{x\}\}$ defined by $f_n(x,t)\to 0$ if $\deg_tf_n(x,t)\to
\infty$ as $n\to\infty$. See \cite[p.~7]{ec1}.

By equation~\eqref{eq:cf}, we have
  \beas g(n,k) & = & [x_1^k\cdots x_n^k]\frac{1}{1-\sum x_i+\sum
      x_ix_{i+1}}\\ & = & [x_1^k\cdots x_n^k]
      \sum_{r\geq 0}\left( \sum x_i-\sum x_ix_{i+1}\right)^r. \eeas 
We obtain a term $\tau$ in the expansion of $\left( \sum x_i-\sum
  x_ix_{i+1}\right)^r$ by picking a term $x_i$ or $-x_ix_{i+1}$ from
each factor. Associate with $\tau$ the graph $G_\tau$ on the vertex
set $[n]$ where we 
put a loop at $i$ every time we choose the term $x_i$, and we put an
edge between $i$ and $i+1$ whenever we choose the term
$-x_ix_{i+1}$. Thus $G_\tau$ is regular of degree $k$ with $r$ edges,
and each connected component has a vertex set which is an interval
$\{a,a+1,\dots, a+b\}$. Let $\mu(i)$ be the number of loops at vertex
  $i$ and $\mu(i,i+1)$ the number of edges between vertices $i$ and
  $i+1$. Let $\nu=\sum \mu(i,i+1)$, the total number of nonloop
  edges. Then 
  \beq [\tau] \left( \sum x_i-\sum x_ix_{i+1}\right)^r =
    \frac{(-1)^\nu r!}{\prod_{i=1}^r \mu_i!\cdot \prod_{i=1}^{r-1}
     \mu(i,i+1)!}. \label{eq:tau} \eeq
Define the \emph{umbralized weight} $w(G)$ of $G=G_\tau$ by
  $$ w(G) = \frac{(-1)^\nu t^r}{\prod_{i=1}^r \mu_i!\cdot
    \prod_{i=1}^{r-1} \mu(i,i+1)!}. $$
Thus $w(G)$ is just the right-hand side of equation~\eqref{eq:tau}
with the numerator factor $r!$ replaced by $t^r$. At the end of the
proof we will ``deumbralize'' by applying the functional $\varphi$,
thus replacing $t^r$ with $r!$. 

Regarding $k$ as fixed, let 
  \beq c(m) = \sum_H w(H), \label{eq:cm} \eeq
summed over all \emph{connected graphs} $H$ on a linearly ordered
$m$-element vertex set, say $[m]$, that are regular of degree $k$ and
such that every edge is either a loop or is between two consecutive
vertices $i$ and $i+1$. It is easy to see by transfer-matrix arguments
(as discussed in \cite[{\S}4.7]{ec1}) that 
  $$  F(x,t):=\sum_{m\geq 1}c(m)x^m $$ 
is a rational function of $x$ whose coefficients are integer
polynomials in $t$. The point is that we can build up $H$ one vertex
at a time in the order $1,2,\dots,m$, and the information we need to
see what new edges are allowed at vertex $i$ (that is, loops at $i$ and
edges between $i$ and $i+1$) depends only on a \emph{bounded} amount
of prior information (in fact, the edges at $i-1$). Moreover, the
contribution to the umbralized weight $w(G)$ from adjoining vertex $i$
and its incident edges is simply the weight obtained thus far
multiplied by $(-1)^{\mu(i,i+1)} (t/2)^{\mu(i)+\mu(i,i+1)}$. Hence we
are counting weighted walks 
on a certain edge-weighted graph with vertex set $[k-1]$ (the possible
values of $\mu(i,i+1)$) with certain initial conditions. We have to
add a term for one exceptional graph: the graph with two vertices and
$k$ edges between them. This extra term does not affect rationality.

We may describe the vertex sets of the connected components of
$G_\tau$ by a composition $(\alpha_1,\dots,\alpha_s)$ of $n$, i.e., a
sequence of positive integers summing to $n$. Thus the vertex sets are
 $$ \{1,\dots,\alpha_1\}, \{\alpha_1+1,\dots,\alpha_1+\alpha_2\},
    \dots, \{\alpha_1+\cdots+\alpha_{s-1}+1,\dots,n\}. $$
Now set
  $$ f(n) = \sum_G w(G), $$
summed over \emph{all} graphs $G$ on a linearly ordered
$n$-element vertex set, say $[n]$, that are regular of degree $k$ and
such that every edge is either a loop or is between two consecutive
vertices $i$ and $i+1$. Thus
  $$ f(n)=\sum_{(\alpha_1,\dots,\alpha_r)}g(\alpha_1)\cdots
     g(\alpha_s), $$ 
where the sum ranges over all compositions of $n$. Note the crucial
fact that if $G$ has $r$ edges and the connected components have
$r_1,\dots,r_s$ edges, then $t^r=t^{r_1+\cdots+r_s}$. It therefore
follows that if 
  $$ G(x,t) = \sum_{n\geq 0}f(n)x^n, $$
then
  $$ G(x,t)=\frac{1}{1-F(x,t)}. $$
It follows from equation~\eqref{eq:tau} that
  \beq \sum_{n\geq 0} g(n,k)x^n = \varphi G(x,t), \label{eq:gxtthm}
  \eeq
where $\varphi$ is the linear functional mentioned above which is
defined by $\varphi(t^r)=r!$. 

By Puiseux' theorem we can write
  $$ G(x,t) = \frac{1}{(1-y_1t)^{d_1}\cdots (1-y_qt)^{d_q}} $$
for certain distinct algebraic fractional Laurent series $y_i\in\afra$ 
and integers $d_j\geq 1$. By partial fractions we have
  \bea G(x,t) & = & \sum_{j=1}^q\sum_{i=1}^{d_j} 
          \frac{u_{ij}}{(1-y_jt)^i} \nonumber \\ & = &
      \sum_{j=1}^q\sum_{i=1}^{d_j}u_{ij}
       \sum_{m\geq 0}\binom{i+m-1}{i} 
        y_j^mt^m \nonumber \\ & = & \sum_{j=1}^q \sum_{m\geq 0}
      \sum_{j=1}^q \sum_{m\geq 0} P_j(m)z_jy_j^m t^m, 
        \label{eq:gxt} \eea
where $u_{ij},z_i\in\afra$ and $P_j(m)\in\afra[m]$. Apply the
functional  $\varphi$ to complete the proof. 
\end{proof}


\begin{example}
Consider the case $k=2$. Let $G$ be a connected regular graph
of degree 2 with vertex set $[m]$ and edges that are either loops or
are between vertices $i$ and $i+1$ for some $1\leq i\leq m-1$. Then
$G$ is either a path with vertices $1,2,\dots, m$ (in that order) with
a loop at both ends (allowing a double loop when $m=1$), or a double
edge when $m=2$. Hence
   \beas F(x,t) & = & \frac 12 t^2x+(\frac 12
   t^2-t^3)x^2+t^4x^3-t^5x^4+\cdots\\ & = & \frac 12t^2x -
   (t^3-\frac 12 t^2)x^2+\frac{t^4x^3}{1+tx}, \eeas
and
 \beas G(x,t) & = & \frac{1}{1-F(x,t)}\\ & = &\frac{1+tx}
  {1+xt-\frac 12(x+x^2)t^2+\frac 12(x^2-x^3)t^3}. \eeas
  The denominator of $G(x,t)$ factors as $(1-y_1t)(1-y_2t)(1-y_3t)$,
  where
  \beas y_1 & = & x+2x^2+18x^3+194x^4+2338x^5+30274x^6+411698x^7\\ 
     & & \qquad + 5800066x^8+\cdots\\[1em]
    y_2 & = & \frac 12\sqrt{2}x^{1/2}-x-\frac 14\sqrt{2}x^{3/2}-x^2-
     \frac{33}{16}\sqrt{2}x^{5/2}-9x^3\\ & & \qquad
     -\frac{657}{32}\sqrt{2} x^{7/2}-97x^4-\cdots\\[1em] y_3 & = &
    -\frac 12\sqrt{2}x^{1/2}-x+\frac 14\sqrt{2}x^{3/2}-x^2+
     \frac{33}{16}\sqrt{2}x^{5/2}-9x^3\\ & & \qquad
       +\frac{657}{32}\sqrt{2}x^{7/2}-97x^4+\cdots. \eeas
Since the $y_i$'s are distinct we can take each $P_i(m)=1$.
The coefficients $z_1,z_2,z_3$ are given by
   \beas z_1 & = & 
   -4x-48x^2-676x^3-10176x^4-158564x^5-2523696x^6+\cdots\\[1em] 
     z_2 & = & \frac 12+\frac 12\sqrt{2}x^{1/2}+2x+\frac{19}{4}
     \sqrt{2}x^{3/2}+24x^4+\frac{1007}{16}\sqrt{2}x^{5/2}+
     338x^3\\ & & \qquad +\frac{29507}{32}\sqrt{2}x^{7/2}+
      \cdots\\[1em] 
     z_3 & = & \frac 12-\frac 12\sqrt{2}x^{1/2}+2x-\frac{19}{4}
     \sqrt{2}x^{3/2}+24x^4-\frac{1007}{16}\sqrt{2}x^{5/2}+
     338x^3\\ & & \qquad -\frac{29507}{32}\sqrt{2}x^{7/2}
       +\cdots. \eeas
Finally we obtain
  $$ \sum_{n\geq 0}g(n,2)x^n = \sum_{m\geq 0}m!(z_1y_1^m +
      z_2y_2^m + z_3y_3^m). $$ 
\end{example}


In Theorem~\ref{thm:fn} we determined the number of equivalence
classes of the equivalence relation $\sim$ on $\sn$. Let us now turn
to the structure of the individual equivalence classes. Given
$w\in\sn$, write $\langle w\rangle$ for the class containing
$w$. First we consider the case where $w$ is the identity permutation
id$_n=12\cdots n$ or its reverse $\overline{\mathrm{id}}_n=n\cdots
21$. (Clearly for any $w$ and its reverse $\bar{w}$ we have $\#\langle
w\rangle =\#\langle \bar{w}\rangle$.) Let $F_n$ denote the $n$th
Fibonacci number, i.e., $F_1=F_2=1$, $F_n=F_{n-1}+F_{n-2}$. 

\begin{prop} \label{prop:idclass}
We have $\#\langle\mathrm{id}_n\rangle=
\#\langle\overline{\mathrm{id}}_n\rangle = F_{n+1}$.
\end{prop}

\begin{proof}
Let $g(n)=\#\langle\mathrm{id}_n\rangle=
\#\langle\overline{\mathrm{id}}_n\rangle$, so $g(1)=1=F_2$ and
$g(2)=2=F_3$. If $w=a_1a_2\cdots a_n\sim\mathrm{id}$, then either
$a_n=n$ with $g(n-1)$ possibilities for $a_1a_2\cdots a_{n-1}$, or else
$a_{n-1}=n$ and $a_n={n-1}$ with $g(n-2)$ possibilities for
$a_1a_2\cdots a_{n-2}$. Hence $g(n)=g(n-1)+g(n-2)$, and the proof
follows.
\end{proof}

We now consider an arbitrary equivalence
class. Proposition~\ref{prop:fibprod} below is due to Joel Lewis
(private communication).

\begin{lemma} \label{lemma:fibprod}
Each equivalence class $\langle w\rangle$ of permutations of any
finite subset $S$ of $\{1,2,\dots\}$ contains a permutation
$v=v_1v_2 \cdots v_k$ (concatenation of words) such that (a) each
$v_i$ is an increasing or decreasing sequence of consecutive integers,
and (b) every $u\sim w$ has the form $u=v'_1 v'_2\cdots v'_k$, where
$v'_i\sim v_i$. Moreover, the permutation $v$ is unique up to
reversing $v_i$'s of length two.
\end{lemma}

\begin{proof}
Let $j$ be the largest integer for which some $v\sim w$ has the
property that $v_1 v_2\cdots v_j$ is either an increasing or
decreasing sequence of consecutive integers. It is easy to see that
$v_k$ for $k>j$ can never be interchanged with some $v_i$ for $1\leq
i\leq j$ in a sequence of transpositions of adjacent consecutive
integers. Moreover, $v_1 v_2\cdots v_j$ cannot be converted to the
reverse $v_j \cdots v_2 v_1$ unless $j\leq 2$. The result follows by
induction. 
\end{proof}

\begin{prop} \label{prop:fibprod}
Let $m_i$ be the length of $v_i$ in Lemma~\ref{lemma:fibprod}. Then
  $$ \#\langle w\rangle = F_{m_1+1}\cdots F_{m_k+1}. $$
\end{prop}

\begin{proof}
Immediate from Proposition~\ref{prop:idclass} and
Lemma~\ref{lemma:fibprod}. 
\end{proof}

We can also ask for the number of equivalence classes of a given size
$r$. Here we consider $r=1$. Let $N(n)$ denote the number of
one-element equivalence classes of permutations in $\sn$. Thus $N(n)$
is also the number of permutations $a_1a_2\cdots a_n\in\sn$ for which
$|a_i-a_{i+1}| \geq 2$ for $1\leq i\leq n-1$. This problem is
discussed in OEIS \cite[A002464]{oeis}. In particular, we have the
generating function
  \beas \sum_{n\geq 0}N(n)x^n & = & \sum_{m\geq 0} m!\left(
    \frac{x(1-x)}{1+x}\right )^m\\ & = & 1+x+2x^4+14x^5+90x^6+646x^7+
     5242x^8+\cdots. \eeas

\section{Multiplicity-free flag $h$-vectors of distributive lattices.}
Let $P$ be a finite graded poset of rank $n$ with $\hat{0}$ and
$\hat{1}$, and let $\rho$ be the rank function of $P$. (Unexplained
poset terminology may be found in \cite[Ch.~3]{ec1}.) Write $2^{[n-1]}$
for the set of all subsets of $[n-1]$. The \emph{flag
  $f$-vector} of $P$ is the function $\alpha_P\colon 2^{[n-1]}\to
  \zz$ defined as follows: if $S\subseteq [n-1]$, then $\alpha_P(S)$
  is the number of chains $C$ of $P$ such that $S=\{\rho(t)\colon t\in
  C\}$. For instance, $\alpha_P(\emptyset)=1$, $\alpha_P(i)$ (short
  for $\alpha_P\{i\})$) is the number of elements of $P$ of rank $i$,
  and $\alpha_P([n-1])$ is the number of maximal chains of $P$. Define
  the \emph{flag $h$-vector} $\beta_P\colon 2^{[n-1]} \to \zz$ by
    $$ \beta_P(S) = \sum_{T\subseteq S}(-1)^{\#(S-T)}\alpha_P(T). $$
 Equivalently, 
    $$ \alpha_P(S) = \sum_{T\subseteq S}\beta_P(T). $$
 We say that $\beta_P$ is \emph{multiplicity-free} if $\beta_P(S)=
 0,\pm 1$ for all $S\subseteq [n-1]$. 

 In this section we will classify and enumerate all $P$ for which
 $\beta_P$ is multiplicity-free. First we consider the case when $P$
 is a distributive lattice, so $P=J(Q)$ (the lattice of order ideals
 of $Q$) for some $n$-element poset $Q$ (see
 \cite[Thm.~3.4.1]{ec1}). Suppose that $Q$ is a natural partial
 ordering of $[n]$, i.e., if $i<j$ in $Q$, then $i<j$ in $\zz$. We may
 regard a \emph{linear extension} of $Q$ as a permutation $w=a_1\cdots
 a_n\in\sn$ for which $i$ precedes $j$ in $w$ if $i<j$ in $Q$. Write
 $\mathcal{L}(Q)$ for the set of linear extensions of $Q$, and let
   $$ D(w) = \{ i\colon a_i>a_{i+1}\}\subseteq [n-1], $$
 the \emph{descent set} of $w$. A basic result in
 the theory of $P$-partitions \cite[Thm.~3.13.1]{ec1} asserts that
   $$ \beta_P(S) = \#\{w\in\mathcal{L}(Q)\colon D(w)=S\}. $$
 It will follow from our results that the equivalence class containing
 $12\cdots n$ of the equivalence relation $\sim$ on $\sn$ is the set
 $\mathcal{L}(Q)$ for a certain natural partial ordering $Q$ of $[n]$ for
 which $\beta_{J(Q)}$ is multiplicity-free, and
 that $Q$ has the most number of linear extensions of any $n$-element
 poset for which $\beta_{J(Q)}$ is multiplicity-free.

In general, if we have a partially commutative monoid $M$ with
generators $g_1,\dots, g_n$ and if $w=g_{i_1}g_{i_2}\cdots g_{i_r}\in
M$, then the set of all words in the $g_i$'s that are equal to $w$
correspond to the linear extensions of a poset $Q_w$ with elements
$1,\dots,r$ \cite[Exer.~3.123]{ec1}. Namely, if $1\leq a<b\leq r$ in
$\zz$, then let $a<b$ in $Q_w$ if $g_{i_a}=g_{i_b}$ or if
$g_{i_a}g_{i_b} \neq g_{i_b}g_{i_a}$. In the case $g_i=i$ and
$w=12\cdots n$, then the set of all words equal to $w$ are themselves
the linear extensions of $Q_w$. For instance, if $n=5$, $g_i=i$, and
the commuting relations are $12=21$, $23=32$, $34=43$, and $45=54$,
then the poset $Q_{12345}$ is shown in Figure~\ref{fig:12345}. The
linear extensions are the words equivalent to 12345 under $\sim$,
namely 12345, 12354, 12435, 21345, 21354, 21435, 13245, 13254. Write
$Q_n$ for the poset $Q_{12\cdots n}$. Define a subset $S$ of $\zz$ to
be \emph{sparse} if it does not contain two consecutive integers.

\begin{figure}
 \centerline{\includegraphics[width=3cm]{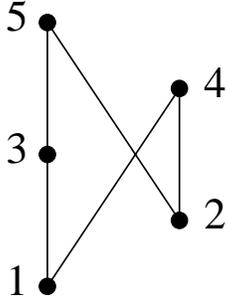}}
\caption{The poset $Q_{12345}$}
\label{fig:12345}
\end{figure}

\begin{lemma} \label{lemma:ds1n}
For each sparse $S\subset [n-1]$, there is exactly one $w\in \ec$ (the
equivalence class containing $12\cdots n$ of the equivalence relation
$\sim$) satisfying $D(w)=S$. Conversely, if $w\in\ec$ then $D(w)$ is
sparse. 
\end{lemma}

\begin{proof}
  The permutations $w\in\ec$ are obtained by taking the identity
  permutation $\mathrm{id}=12\cdots n$, choosing a sparse subset
  $S\subset [n-1]$, and transposing $i$ and $i+1$ in id when $i\in
  S$. The proof follows.
\end{proof}

\begin{prop} \label{prop:maxep}
Let $Q$ be an $n$-element poset for which the flag $h$-vector of
$J(Q)$ is multiplicity-free. Then $e(Q)\leq F_{n+1}$ (a Fibonacci
number). Moreover, the unique such poset (up to isomorphism) for which
equality holds is $Q_n$.
\end{prop}

We will prove Proposition~\ref{prop:maxep} as a consequence of a
stronger result: the complete classification of all posets $Q$ for which
the flag $h$-vector of $J(Q)$ is multiplicity-free. The key
observation is the following trivial result.

\begin{lemma} \label{lemma:2el}
Let $P$ be any graded poset whose flag $h$-vector $\beta_P$ is
multiplicity-free. Then $P$ has at most two elements of each rank. 
\end{lemma}

\begin{proof}
Let $P$ have rank $n$ and $1\leq i\leq n-1$. Then $\beta_P(i)
=\alpha_P(i)-1$, i.e., one less than the number of elements of rank
$i$. The proof follows.
\end{proof}

Thus we need to consider only distributive lattices $J(Q)$ of rank $n$
with at most two elements of each rank. A poset $Q$ is said to be
\emph{$(\bm{2}+\bm{2})$-free} if it does not have an induced subposet
isomorphic to the disjoint union of two 2-element chains. Such a poset
is also an \emph{interval order}
\cite{fishburn}\cite[Exer.~3.15]{ec1}\cite{trotter}. Similarly a poset
is \emph{$(\bm{1}+\bm{1}+\bm{1})$-free} or of \emph{width} at most two
if it does not have a 3-element antichain.

\begin{thm} \label{thm:distlat}
Let $Q$ be an $n$-element poset. The following conditions are
equivalent.
 \be\item[(a)] The flag $h$-vector $\beta_{J(Q)}$ of $J(Q)$ is
 multiplicity-free (in which case if $\beta_{J(Q)}(S)\neq 0$, then $S$
 is sparse).
  \item[(b)] For $0\leq i\leq n$, $J(Q)$ has at most two elements of
    rank $i$. Equivalently, $Q$ has at most two $i$-element order
    ideals. 
  \item[(c)] $Q$ is $(\bm{2}+\bm{2})$-free and
    $(\bm{1}+\bm{1}+\bm{1})$-free. 
  \ee
Moreover, if $f(n)$ is the number of nonisomorphic $n$-element posets
satisfying the above conditions, then
  $$ \sum_{n\geq 0}f(n)x^n = \frac{1-2x}{(1-x)(1-2x-x^2)}. $$
If $g(n)$ is the number of such posets which are not a nontrivial
ordinal sum (equivalently, $J(Q)$ has exactly two elements of each
rank $1\leq i\leq n-1$), then 
  \beq g(1)=g(2)=1\ \mathrm{and}\ g(n)=2^{n-3},\ n\geq 3.
    \label{eq:gn} \eeq 
\end{thm}

\begin{proof}
  Consider first a distributive lattices $J(Q)$ of rank $n$ with
  exactly two elements of rank $i$ for $1\leq i\leq n-1$. Such
  lattices are described in \cite[Exer.~3.35(a)]{ec1}. There is one
  for $n\leq 3$. The unique such lattice of rank three is shown in
  Figure~\ref{fig:rank3}. Once we have such a lattice $L$ of rank
  $n\geq 3$, we can obtain two of rank $n+1$ by adjoining an element
  covering the left or right coatom (element covered by $\hat{1}$) of
  $L$ and a new maximal element, as
  illustrated in Figure~\ref{fig:rank4} for $n=3$. When $n=1$ (so $L$
  is a 2-element chain), we say by convention that the $\hat{0}$ of
  $L$ is a \emph{left} coatom. When $n=2$ (so $L$ is the
  boolean algebra $B_2$) we obtain isomorphic posets by adjoining an
  element covering either the left or right coatom, so again by
  convention we always choose the \emph{right} coatom. 
  Thus every such $L$
  of rank $n\geq 1$ can be described by a word
  $\gamma=\gamma_1\gamma_2\cdots \gamma_{n-1}$, where $\gamma_1=0$
  (when $n\geq 2$), $\gamma_2=1$ (when $n\geq 3$), and $\gamma_i=0$ or
  1  for $i\geq 3$. 
  If $\gamma_i=0$, then at rank $i$ we adjoin an element on the
  left, otherwise on the right.  Write $L(\gamma)$ for this
  lattice. Figure~\ref{fig:llr} shows $L(01001)$.

\begin{figure}
 \centerline{\includegraphics[width=3cm]{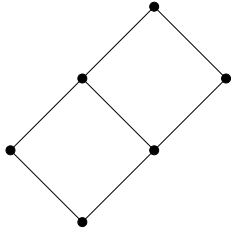}}
\caption{A distributive lattice of rank three}
\label{fig:rank3}
\end{figure}

\begin{figure}
 \centerline{\includegraphics[width=6cm]{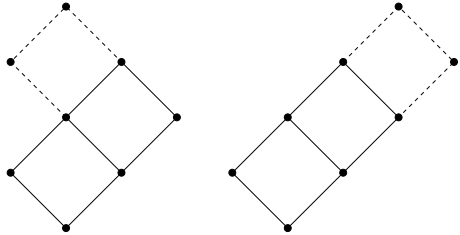}}
\caption{Extending a lattice of rank three}
\label{fig:rank4}
\end{figure}

\begin{figure}
 \centerline{\includegraphics[width=3cm]{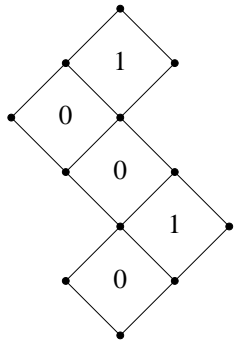}}
\caption{The distributive lattice $L(01001)$}
\label{fig:llr}
\end{figure}

Suppose that $n\geq 3$ and $\gamma$ has the form $\delta j i^r$, where
$r\geq 1$, $i=0$ or 1, and $j=1-i$. For example, if $\gamma=0111$ then
$\delta= \emptyset$ (the empty word) and $r=3$. If $\gamma=0101100$,
then $\delta=0101$ and $r=2$. Consider the lattice $L=L(\delta j
i^r)$. Then for one coatom $t$ of $L$ we have $[\hat{0},t]\cong
L(\delta j i^{r-1})$. For the other coatom $u$ of $L$ there is an
element $v<u$ such that $[v,u]$ is a chain and $[\hat{0},v] \cong
L(\delta)$. It follows easily that for $S\subseteq [n-1]$ we have the
recurrence
  \beq \beta_{L(\delta ij^r)}(S) = \left\{ \begin{array}{rl}
      \beta_{L(\delta ij^{r-1})}(S), & n-1\not\in S\\[.5em]
       \beta_{L(\delta)}(S), & n-1\in S. \end{array} \right. 
    \label{eq:betarec} \eeq
Hence by induction $\beta_L$ is multiplicity-free. Morever, if
$\beta_L(S)\neq 0$, then $S$ is sparse. 

Suppose now that $J(Q)$ has exactly one element $t$ of some rank
$1\leq i\leq n-1$. Let $[\hat{0},t]\cong J(Q_q)$ and $[t,\hat{1}]
\cong J(Q_2)$. Then $Q=Q_1\oplus Q_2$ (ordinal sum), and 
   \beq \beta_{J(Q)}(S) = \beta_{J(Q_1)}(S\cap [i-1])\cdot
      \beta_{J(Q_1)}(S'\cap [n-i-1]), \label{eq:ordsum} \eeq
where $S'=\{ j\colon i+j\in S\}$. \marginpar{??} In particular,
$\beta_{J(Q)}(S)=0$ 
if $i\in S$. It follows from Lemma~\ref{lemma:2el} and
equations~\eqref{eq:betarec} and \eqref{eq:ordsum} that (a) and (b)
are equivalent.  

Let us now consider condition (c). One can easily check that if
$L(\gamma)=J(Q(\gamma))$, then $Q(\gamma)$ is $(\bm{2}+\bm{2})$-free
and $(\bm{1}+\bm{1}+\bm{1})$-free. (Alternatively, if a poset $Q$
contains an induced $\bm{2}+\bm{2}$ or $\bm{1}+\bm{1}+\bm{1}$ then
it contains them as a convex subset, i.e., as a subset $I-I'$ where
$I\leq I'$ in $J(Q)$. By considering linear extensions of $Q$ that
first use the elements of $I'$ and then those of $I'-I$, one sees that
at least two linear extensions have the same descent set.) Thus
(b)$\Rightarrow$(c). 

Conversely, suppose that $J(Q)$ has three elements of the same rank
$i$. It is easy to see that the restriction of $J(Q)$ to ranks $i$ and
$i+1$ is a connected bipartite graph. If an element $t$ of rank $i+1$
covers at least three elements of rank $i$, then $Q$ contains an
induced $\bm{1}+\bm{1}+\bm{1}$. Otherwise there must be elements $s,t$
of rank $i+1$ and $u,v,w$ of rank $i$ for which $u,v<s$ and
$v,w<t$. The interval $[u\wedge v\wedge w,u\vee v\vee w]$ is either
isomorphic to a boolean algebra $B_3$, in which case $Q$ contains an
induced $\bm{1}+\bm{1}+\bm{1}$, or to $\bm{3}\times\bm{3}$, in which
case $P$ contains an induced $\bm{2}+\bm{2}$. Hence
(c)$\Rightarrow$(b), completing the proof of the equivalence of (a),
(b), and (c). 

We have already observed that $g(n)$ is given by
equation~\eqref{eq:gn}. Thus
  $$ A(x):= \sum_{n\geq 1} g(n)x^n = x+x^2+\frac{x^3}{1-2x}. $$
Elementary combinatorial reasoning shows that
  \beas \sum_{n\geq 0}f(n)x^n & = & \frac{1}{1-A(x)}\\
     & = & \frac{1-2x}{(1-x)(1-2x-x^2)}, \eeas
completing the proof. 
\end{proof}

\begin{cor} \label{cor:maxep}
Let $Q$ be an $n$-element poset for which $\beta_{J(Q)}$ is
multiplicity-free. Then $e(Q)\leq F_{n+1}$ (a Fibonacci number), with 
equality if and only if $Q=Q(0101\cdots)$ where $0101\cdots$ is an
alternating sequence of $n-1$ 0's and 1's. 
\end{cor}

\begin{proof}
There are $F_{n+1}$ sparse subsets of $[n-1]$. It follows from the
parenthetical comment in Theorem~\ref{thm:distlat}(a) that $e(Q)\leq
F_{n+1}$. Moreover, equations~\eqref{eq:betarec} and
\eqref{eq:ordsum} make it clear that $\beta_{J(Q)}(S) =1$ for all
sparse $S\subset [n-1]$ if and only if $Q=Q(0101\cdots)$, so the proof
follows. 
\end{proof}

\begin{prop}
The $n$-element poset $Q(0101\cdots)$ is isomorphic to $Q_{12\cdots
  n}$. Hence The set $\mathcal{L}(Q(0101\cdots)$ of linear extensions
of $Q(0101\cdots)$ is equal to the equivalence class $\langle 12\cdots 
n\rangle$. 
\end{prop}

\begin{proof}
Immediate from Lemma~\ref{lemma:ds1n}.
\end{proof}

\section{Multiplicity-free flag $h$-vectors of graded posets.}

We now consider \emph{any} graded poset $P$ of rank $n$ for which
$\beta_P$ is multiplicity-free. By Lemma~\ref{lemma:2el} there are at
most two elements of each rank. If there is just one element $t$ of
some rank $1\leq i\leq n-1$, then let $P_1=[\hat{0},t]$ and
$P_2=[t,\hat{1}]$. Equation~\eqref{eq:ordsum} generalizes
easily to
  $$ \beta_P(S) =\left\{ \begin{array}{rl}
      0, & i\in S\\ \beta_{P_1}(S\cap [i-1])\cdot 
       \beta_{P_2}(S'\cap [n-i-1], i\not\in S, \end{array} \right. $$ 
where $S_1=\{j\colon i+j\in S\}$. Hence $\beta_P$ is multiplicity-free
if and only if both $\beta_{P_1}$ and $\beta_{P_2}$ are
multiplicity-free. 

By the previous paragraph we may assume that $P$ has exactly two
elements of each rank $1\leq i\leq n-1$, i.e., of each \emph{interior
  rank}. There are up to isomorphism three possibilities for the
restriction of $P$ to two consecutive interior ranks $i$ and $i+1$
($1\leq i\leq n-2$). See Figure~\ref{fig:tworanks}. If only type (b)
occurs, then we obtain the distributive lattice $L(\gamma)$ for some
$\gamma$. Hence all graded posets with two elements of each interior
rank can be obtained from some $L(\gamma)$ by a sequence of replacing
the two elements of some interior rank with one of the posets in
Figure~\ref{fig:tworanks}(a) or (c). 

\begin{figure}
 \centerline{\includegraphics[width=6cm]{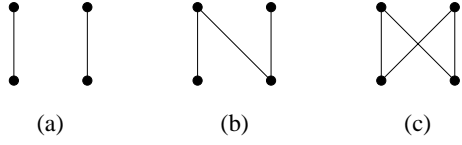}}
\caption{Two consecutive ranks}
\label{fig:tworanks}
\end{figure}

First consider the situation where we replace the two elements of some
interior rank of $P$ with the poset of
Figure~\ref{fig:tworanks}(a). We can work with the following somewhat
more general setup.  Let $R$ be any graded poset of rank $n$ with
$\hat{0}$ and $\hat{1}$. For $1\leq i\leq n-1$ let $R[i]$ denote the
\emph{stretching} of $R$ at rank $i$, namely, for each element $t\in
R$ of rank $i$, adjoin a new element $t'>t$ such that $t'<u$ whenever
$t<u$ (and no additional relations not implied by these conditions).
Figure~\ref{fig:a} shows an example.
Regarding $i$ as fixed, let $S\subset [n]$. If not both $i,i+1\in S$
then let $S^\circ$ be obtained from $S$ by replacing each element
$j\in S$ satisfying $j\geq i+1$ with $j-1$. On the other hand, if both
$i,i+1\in S$ then let $S^\circ$ be obtained from $S$ by removing $i+1$
and replacing each $j\in S$ such that $j>i+1$ with $j-1$. It is easily
checked that
  $$ \beta_{R[i]}(S) =\left\{ \begin{array}{rl}
    \beta_R(S^\circ), & \mathrm{if\ not\ both}\ i,i+1\in S\\[.5em]
       -\beta_R(S^\circ), & \mathrm{if\ both}\ i,i+1\in S.
     \end{array} \right. $$
It follows immediately that if $\beta_R$ is multiplicity-free, then so
is $\beta_{R[i]}$. 

\begin{figure}
 \centerline{\includegraphics[width=6cm]{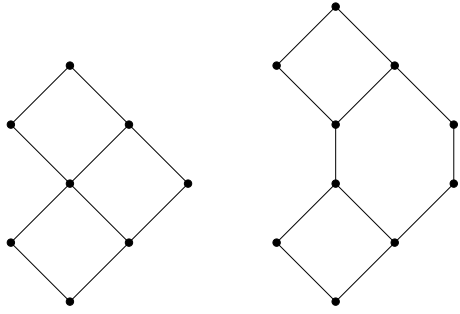}}
\caption{An example of $R$ and its stretching $R[2]$}
\label{fig:a}
\end{figure}

Now consider the situation where we replace the two elements of some
interior rank of $P$ with the poset of Figure~\ref{fig:tworanks}(c).
Again we can work in the generality of any graded poset $R$ of rank
$n$ with $\hat{0}$ and $\hat{1}$. If $1\leq i\leq n-1$, let $R\langle
i\rangle$ denote the \emph{proliferation} of $R$ at rank $i$, namely,
for each element $t\in R$ of rank $i$, adjoin a new element $t'>s$ for
every $s$ of rank $i$ such that $t'<u$ whenever $t<u$ (and no
additional relations not implied by these
conditions). Figure~\ref{fig:c} shows an example. Note that if $R_1$
denotes the restriction of $R$ to ranks $0,1,\dots, i$, and if $R_2$
denotes the restriction of $R$ to ranks 
$i,i+1,\dots,n$, then $R\langle i\rangle = R_1\oplus R_2$ (ordinal
sum). Let $\bar{R}_1$ denote $R_1$ with a $\hat{1}$ adjoined and
$\bar{R}_2$ denote $R_2$ with a $\hat{0}$ adjoined. It is then clear  
(in fact, a simple variant of equation~\eqref{eq:ordsum}) that
  $$ \beta_{R\langle i\rangle}(S) =  \beta_{R_1}(S\cap [i])\cdot
     \beta_{R_2}(S'), $$
where $S'=\{j\colon i+j\in S\}$. \marginpar{??}

\begin{figure}
 \centerline{\includegraphics[width=6cm]{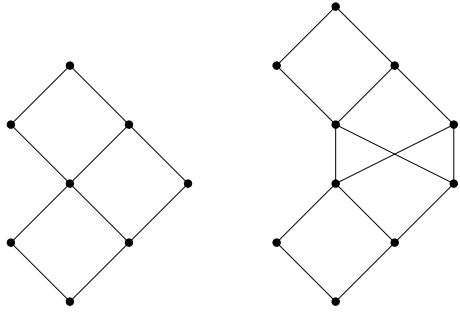}}
\caption{An example of $R$ and its proliferation $R\langle 2\rangle$}
\label{fig:c}
\end{figure}

We have therefore proved the following result.

\begin{thm} \label{thm:genmf}
Let $P$ be a finite graded poset with $\hat{0}$ and $\hat{1}$. The
following conditions are equivalent:
  \be\item[(a)] The flag $h$-vector $\beta_P$ is multiplicity-free.
   \item[(b)] $P$ has at most two elements of each rank.
  \ee
\end{thm}

The above description of graded posets with multiplicity-free flag
$h$-vectors allows us to enumerate such posets.

\begin{thm} \label{thm:mfenum}
Let $h(n,k)$ denote the number of $k$-element graded posets $P$ of
rank $n$ with $\hat{0}$ and $\hat{1}$ for which $\beta_P$ is
multiplicity-free. Let
  $$ U(x,y) = \sum_{n\geq 1}\sum_{k\geq 2}h(n,k)x^ky^n. $$
Then
   $$ U(x,y) =\frac{xy^2(1-xy^2)(1-3xy^3)} 
     {1-xy-5xy^2+4x^2y^3+5x^2y^4-3x^3y^5}. $$
\end{thm}

\begin{proof}
The factor $xy^2$ in the numerator accounts for the $\hat{0}$ and
$\hat{1}$ of $P$. Write $P'=P-\{\hat{0},\hat{1}\}$. We first consider
those $P'$ that are not an ordinal sum of smaller nonempty
posets. These 
will be the one-element poset $\bm{1}$ and posets for which every rank
has two elements, with the restrictions to two consecutive ranks given
by Figure~\ref{fig:tworanks}(a,b). We obtain $P'$ by first choosing a
poset $R$ whose consecutive ranks are given by
Figure~\ref{fig:tworanks}(b) and then doing a sequence of
stretches. By equation~\eqref{eq:gn}, the number of ways to choose $R$
with $m$ levels is 1 for $m=1$ and $2^{m-2}$ for $m\geq 2$. We can
stretch such an $R$ by choosing a sequence $(j_1,\dots,j_m)$ of
nonnegative integers and stretching the $i$th level of $R$ $j_i$
times. Hence the generating function for the posets $P'$ is given by
  \beas R(x,y) & = &  xy+\frac{xy^2}{1-xy^2} +\sum_{m\geq 2}
      \frac{2^{m-2}(xy^2)^m}{(1-xy^2)^m}\\ & = & 
      xy +\frac{xy^2}{1-xy^2}+\frac{x^2y^4}
      {(1-xy^2)(1-3xy^2)}. \eeas
All posets being enumerated are unique ordinal sums of posets $P'$,
with a $\hat{0}$ and $\hat{1}$ adjoined at the end. Thus
   \beas  U(x,y) & = & \frac{xy^2}{1-R(x,y)}\\
   & = & \frac{xy^2(1-xy^2)(1-3xy^2)} 
     {1-xy-5xy^2+4x^2y^3+5x^2y^4-3x^3y^5}. \eeas
\end{proof}

As special cases, the enumeration by rank of graded posets $P$ with
$\hat{0}$ and $\hat{1}$ for which $\beta_P$ is multiplicity-free is
given by
  \beas R(x,1) & = & \frac{x(1-x)(1-3x)}{1-6x+9x^2-3x^3}\\ & = & 
     x+2x^2+6x^3+21x^4+78x^5+297x^6+1143x^7+4419x^8+\cdots. \eeas
Similarly, if we enumerate by number of elements we get
   \beas R(1,y) & = & 
     \frac{y^2(1-y^2)(1-3y^2)}{1-y-5y^2+4y^3+5y^4-3y^5}\\ & = &
     y^2+y^3+2y^4+3y^5+7y^6+12y^7+28y^8+51y^9+117y^{10}+\cdots. 
    \eeas

\end{document}